\documentclass[reqno]{article}
\usepackage{epsfig}
\usepackage{latexsym}
\usepackage{fullpage}
\usepackage{amsthm,amsmath,amssymb}
\usepackage{graphicx} 
\usepackage{subfigure}
\usepackage{color}
\usepackage{algorithm}
\usepackage{algorithmic}
\usepackage{tikz}
\usetikzlibrary{arrows}

\usepackage{hyperref}

\newcommand{\real}{\mathbb{R}}

\newtheorem{remark}{Remark}

\newtheorem{definition}{Definition}
\newtheorem{assumption}{Assumption}
\newtheorem{lemma}{Lemma}

\newtheorem{corollary}{Corollary}

\newcommand{\A}{\mathcal{A}}

\newcommand{\Sn}{\mathcal{S}^{n-1}}

\newcommand{\rd}{\mathrm{d}}

\newcommand{\ones}{{\bf1}}

\newcommand{\m}{\mathrm{med}}

\newcommand{\R}{\mathbb{R}}

\begin{document}

\date{}

\title{An Adaptive Total Variation Algorithm for Computing the Balanced Cut of a Graph}

\author{ Xavier Bresson\thanks{Department of Computer Science, City University of Hong Kong, Hong Kong ({\tt xbresson@cityu.edu.hk}).},   Thomas Laurent\thanks{Department of Mathematics, University of California Riverside, Riverside CA 92521  ({\tt laurent@math.ucr.edu})}, David Uminsky\thanks{Department of Mathematics, University of San Francisco, San Francisco CA 94117 ({\tt duminsky@usfca.edu})} and James H. von Brecht\thanks{Department of Mathematics, University of California Los Angeles, Los Angeles CA 90095 ({\tt jub@math.ucla.edu})}  }

\maketitle

\abstract{ We propose an adaptive version of the total variation algorithm  proposed in \cite{BLUV12}   for computing the balanced cut of a graph.  The algorithm from \cite{BLUV12} used a sequence of inner total variation minimizations to guarantee descent of the balanced cut energy as well as convergence of the algorithm. In practice the total variation minimization step is never solved exactly.  Instead, an accuracy parameter is specified and the total variation minimization terminates once this level of accuracy is reached. The choice of this parameter can vastly impact both the computational time of the overall algorithm as well as the accuracy of the result. Moreover,  since the total variation minimization step is not solved exactly, the algorithm is not guarantied to be monotonic. In the present  work we introduce a new adaptive stopping condition for the total variation minimization that guarantees monotonicity. This results in an algorithm that is actually monotonic in practice and is also significantly faster than previous, non-adaptive algorithms.  }


\section{Introduction}
Recent works \cite{SB09, pro:SzlamBresson10,pro:HeinBuhler10OneSpec, pro:HeinSetzer11TightCheeger,
art:BertozziFlenner11DiffuseClassif,
art:BressonTaiChanSzlam12TransLearn, pro:Rang-Hein-constrained, VGB12, MKB12, BLUV12}  have exploited advances in total variation minimization, originally developed for applications in image processing, to tackle fundamental problems in machine learning. The total variation of an image, described by a function $f(x,y):[0,1] \times [0,1] \to \mathbb{R},$ is given by
  \begin{equation} \label{TV1}
 \|f\|_{TV}= \int_{[0,1] \times [0,1]} |\nabla f(x,y)| \; \rd x \rd y.
 \end{equation}
 The total variation can also be given a sense in the context of graph theory: given a weighed graph  with vertices  $V=\{x_i, \ldots, x_n\}$ and weights  $\{w_{i,j}\}_{ 1\le i,j\le n}$ on its edges,
  the total variation of a function $f: V \to \real$,  is given by
 \begin{equation} \label{TV2}
 \|f\|_{TV}= \sum_{i,j} w_{ij} |f(x_i)-f(x_j)|.
\end{equation}
 Minimizing   energies involving  \eqref{TV1} or \eqref{TV2} is  challenging  due to the  nonlinear and non-differentiable nature of the problems.  In the past five years however, important mathematical breakthroughs together with faster computers have given rise  to efficient algorithms for total variation minimization \cite{art:GoldsteinOsher09SB,BT09,art:ChambollePock11FastPD}. These advances have opened many possibilities in imaging sciences, and nowadays the total variation functional plays a central role  in image processing for de-noising and segmentation problems. Recent works \cite{SB09, pro:SzlamBresson10,pro:HeinBuhler10OneSpec, pro:HeinSetzer11TightCheeger,
art:BertozziFlenner11DiffuseClassif,
art:BressonTaiChanSzlam12TransLearn, pro:Rang-Hein-constrained, VGB12, MKB12, BLUV12}  have applied total variation techniques in machine learning and demonstrated they represent a set of very promising tools that we broadly refer to as ``Total Variation Clustering.''  

Given a set of data points $V=\{ x_1, \cdots, x_n\}$ and similarity weights $\{w_{i,j}\}_{ 1\le i,j\le n}$ between these data points, the Balance Cut Problem \cite{art:Cheeger70RatioCut, book:Chung97Spectral} is:
  \begin{equation}\label{cc}
 \text{Minimize } \;\;  \mathcal{C}(S):= \frac{ \text{Cut}(S,S^c) }{ \min( |S|, |S^c|)} \qquad
 \text{over all subsets $S\subsetneq V$}. 
\end{equation}
Here the numerator  $\text{Cut}(S,S^c)$ stands for  $\sum_{x_i \in S, x_j \in S^c}  w_{i,j}$,  and the term $|S|$ in the denominator denotes the number of data points in $S$. The balance cut problem \eqref{cc} attempts to partition the dataset into two groups of comparable size that are weakly linked. The Balanced Cut problem is an NP-hard problem. However, several recent works \cite{SB09, BLUV12} have shown that the combinatorial problem \eqref{cc} is equivalent to the following \emph{continuous relaxation}
\begin{align} \label{rc2}
\text{Minimize}  \;\;  E_{}(f) := \frac{||f||_{TV}}{ ||f - \m(f) \ones||_{1} }  \qquad  \text{over all non-constant $f \in \mathbb{R}^n$},
\end{align}
called the TV-Balanced Cut. Here $||f||_{1} = \sum_{i} |f_i|$ denotes the $\ell_1$ norm of $f$ and $\m(f)$ denotes the median of $f$, i.e. the $n/2$ smallest entry when $n$ is even. The problem \eqref{rc2} is non-convex, but is provably equivalent to the original problem. Specifically, a one-to-one correspondence exists between the global minimizers of each problem. Moreover, the continuous problem is much easier to optimize. The lack of convexity means that the resulting optimization can have difficulties with local minima, however.

Several algorithms have appeared that attempt to minimize the TV-Balanced Cut. In this work we propose a new adaptive total variation algorithm that, to the best of our knowledge, provides the fastest and most reliable approach. 
Our previous algorithm \cite{BLUV12} utilized a sequence of ``inner'' total variation minimizations to guarantee descent of the TV-Balanced cut energy as well as convergence of the algorithm:

\begin{algorithm}
\label{algo1}
\caption{TV algorithm for computing the Balanced Cut}
\vspace{.1cm}
\begin{center}
\begin{algorithmic}
\STATE $f^{0} $ non-constant function with $\m(f) = 0$ and $||f^0||_{2} = 1$. \\
\WHILE{$E(f^{k}) - E(f^{k+1})  \geq \mathrm{TOL}$ }
\STATE $v^k \in \partial_{0} ||f^k||_{1}$ \\
\STATE{{$g^{k}=f^{k}+v^k$}}\\
\STATE{$ h^k= \underset{u\in\R^n}{\arg\min}  \{ ||u||_{TV}+\frac{E(f^k)}{2}||u-g^{k}||_2^2  \}$}\\
\STATE{$h^{k}_{0} =h^k - \m(h^k)\ones$}
 \STATE{$f^{k+1}= \frac{h^k_0}{\|h^k_0\|_2} $}
\ENDWHILE
\end{algorithmic}
\end{center}
\end{algorithm}
In practice the total variation minimization step (also known as the ROF problem 
\cite{art:RudinOsherFatemi92ROF}),  
\begin{equation}
h^k= \underset{u\in\R^n}{\arg\min} \left\{ ||u||_{TV}+\frac{E(f^k)}{2}||u-g^{k}||_2^2 \right\}
\label{TVstep}
\end{equation}
is never solved exactly. Instead, a total variation minimization algorithm, such as those proposed in \cite{art:GoldsteinOsher09SB,BT09,art:ChambollePock11FastPD}, will generate a sequence of iterates $\{h^k_i\}_{i=1}^{\infty}$ that converge toward the exact solution $h^k$ defined by \eqref{TVstep}.
An accuracy parameter $\epsilon>0$ is then specified and the total variation minimization algorithm terminates once 
\begin{equation}
\|h_{i+1}^k-h_i^k\|_2 \le \epsilon. \label{old_stopping_criteria}
\end{equation}
%

 The choice of the parameter $\epsilon$ can vastly impact both the computational time of the overall algorithm as well as the accuracy of the result. It remains unclear how to properly choose the level of accuracy to obtain the right balance between these two aims. In addition all theoretical properties of this algorithm, along with any other algorithm proposed for the TV-Balanced Cut, are derived under the assumption that the total variation solution is exactly obtained. They therefore no longer hold in the actual implementation of the algorithm.  The most important of these properties is monotonicity, i.e. that the TV-Balanced Cut energy is guaranteed to decrease $E(f^{k+1}) > E(f^k)$ at every outer iteration.  In this work, we propose an adaptive stopping condition for the total variation minimization that still guarantees monotonicity of the algorithm. This results in an algorithm that is actually monotonic in practice and is more than two times faster on benchmark databases, such as the MNIST database, without sacrificing accuracy of the result. The key idea lies in solving the total variation step only to the amount needed to obtain ``sufficient energy descent,'' where ``sufficient'' has a precise mathematical meaning that guarantees the important theoretical properties of the idealized algorithm still hold. 

\section{The Proposed Algorithm}
We propose to replace the stopping condition \eqref{old_stopping_criteria}, which is used by all TV-Balanced cut algorithms to date  \cite{SB09, pro:SzlamBresson10,pro:HeinBuhler10OneSpec, pro:HeinSetzer11TightCheeger, BLUV12}, by an adaptive stopping condition that guarantees monotonicity and results in a significantly more efficient algorithm overall. The genesis of this idea lies in the following energy inequality
\begin{equation}
E(f^k) \geq E(h^k) + \frac{E(f^k) ||h^k - f^k||^2_2 }{|| h^k - \m(h^k)\ones||_{1} }
\label{eq:inequality}
\end{equation}
that holds for the idealized algorithm above. See \cite{BLUV12} for a proof of this result. This inequality guarantees that the energy $E(f)$ decreases by at least
$$
\frac{E(f^k) ||h^k - f^k||^2_2 }{|| h^k - \m(h^k)\ones||_{1} }
$$
at every iteration. Moreover, this energy inequality forms the basis of the proof for the theoretical properties of the idealized algorithm.  

Our adaptive stopping condition simply uses a relaxed version of this inequality \eqref{eq:inequality}. Fix $\theta \in (0,1)$ and let $\{h^k_i\}_{i=1}^{\infty}$ denote the sequence of iterates generated by a total variation minimization  algorithm solving the inner problem \eqref{TVstep}. Since $\lim_{i\to +\infty} h^k_i = h^k$, we have
\begin{equation} \label{bla}
\lim_{i\to +\infty}\left\{ E(h_i^k) + \frac{\theta E(f^k) ||h_i^k - f^k||^2_2 }{|| h_i^k - \m(h_i^k)\ones||_{1} }\right\} =E(h^k) + \frac{\theta E(f^k) ||h^k - f^k||^2_2 }{|| h^k - \m(h^k)\ones||_{1}} < E(f^k).
\end{equation}
The above equality comes from  the continuity of each of the following: the energy $E$; the median; the $\ell_1$ norm; and the $\ell_2$ norm. That \eqref{bla} holds with strict inequality follows as a consequence of \eqref{eq:inequality} together with the fact that $\theta < 1$. From \eqref{bla} it is clear that for $i$ large enough the following holds:
\begin{equation} \label{new_stopping_criteria}
 E(h_i^k) + \frac{\theta E(f^k) ||h_i^k - f^k||^2_2 }{|| h_i^k - \m(h_i^k)\ones||_{1} }< E(f^k).
\end{equation}
In this work, we propose to use inequality \eqref{new_stopping_criteria}  as the stopping criteria when solving the inner problem \eqref{TVstep}.
This leads to the proposed algorithm:
\begin{algorithm}
\caption{Adaptive TV algorithm for computing the Balanced Cut}
\vspace{.1cm}
\begin{center}
\begin{algorithmic}
\STATE $f^{0} $ non-constant function with $\m(f) = 0$ and $||f^0||_{2} = 1$, $\theta=.99$. \\
\WHILE{$E(f^{k}) - E(f^{k+1})  \geq \mathrm{TOL}$ }
\STATE $v^k \in \partial_{0} ||f^k||_{1}$ \\
\STATE{{$g^{k}=f^{k}+\, v^k$}}\\
\STATE{ Solve $h^k \approx \underset{u\in\R^n}{\arg\min}  \{ ||u||_{TV}+\frac{E(f^k)}{2}||u-g^{k}||_2^2  \}$ until 
	$$E(f^k) > E(h^k) + \frac{\theta\; E(f^k) ||h^k - f^k||^2_2 }{|| h^k - \m(h^k)\ones||_{1} }$$ }\\
\STATE{$h^{k}_{0} =h^k - \m(h^k)\ones$}
 \STATE{$f^{k+1}= \frac{h^k_0}{\|h^k_0\|_2} $}
\ENDWHILE
\end{algorithmic}
\end{center}
\label{algo2}
\end{algorithm}

The notation $v^k \in \partial_{0} ||f^k||_{1}$ means that $v^k$ denotes any element of the \emph{sub-differential} $\partial ||f^k||_{1}$ of the $\ell_1$-norm at $f^k$ that has zero mean. Note that $\partial_{0} ||f^k||_{1}$ is never empty due to the fact that $f^k$ has zero median. Indeed, we can take the particular choice of $v^k \in \R^n$ due to \cite{pro:HeinBuhler10OneSpec},
\begin{equation}
v^k(x_{i}) := \begin{cases} 
\mathrm{sign}( f^k(x_i) ) & \text{ if } f^k(x_i) \neq 0 \\
(n^- -n^+ )/(n_0) & \text{ if } f^k(x_i) = 0
\end{cases},
\end{equation} 
where $n^+$, $n^-$ and $n^0$ denote the number of elements in the sets $\{x_i: f(x_i)>0\}$, $\{x_i: f(x_i)>0\}$ and $\{x_i: f(x_i)=0\}$, respectively. Other possible choices also exist, so that $v^k$ is not uniquely defined. This idea, i.e. choosing an element from the sub-differential with mean zero, was introduced in  \cite{pro:HeinBuhler10OneSpec} and proves indispensable when dealing with median zero functions.

We choose the parameter $\theta$ close to one, e.g. $\theta = 0.99$, in our implementation of the proposed algorithm. We keep $\theta$ strictly smaller than one so that we can guarantee the stopping condition \eqref{new_stopping_criteria} is, in fact, reached in a finite number of iterations. Our experiments have indicated that a larger choice for $\theta$ leads to a more efficient algorithm. In the actual implementation of the algorithm we do not observe any difference between choosing $\theta=0.99$, $\theta=0.999$ or $\theta=0.9999$.

\

The new stopping criterion \eqref{new_stopping_criteria} has three significant advantages over the more traditional stopping criterion \eqref{old_stopping_criteria} used in \cite{SB09, pro:SzlamBresson10,pro:HeinBuhler10OneSpec, pro:HeinSetzer11TightCheeger, BLUV12}.
\begin{enumerate}
\item \underline{Monotonicity}: With the new stopping criterion \eqref{new_stopping_criteria}  the energy $E(f)$ is guaranteed to decrease at every step of the outer loop. In other words, the algorithm as implemented is now truly monotonic. Indeed, the stopping condition \eqref{new_stopping_criteria} was specially designed to achieve this. The fixed, non-adaptive condition \eqref{old_stopping_criteria} simply does not guarantee monotonicity in the implemented algorithm.
\item  \underline{Robustness with Respect to Choice of Parameters}: We observe in our experiments that the adaptive algorithm is not sensitive to choice of the parameter $\theta$ as long as $\theta \approx 1$ and $\theta<1$. Specifically, $\theta = .99$ (or $\theta = .999$) is nearly optimal for any dataset. This markedly contrasts with the old stopping criterion \eqref{old_stopping_criteria};  the non-adaptive algorithm is very sensitive to the choice of the parameter $\epsilon$ in terms of both accuracy and efficiency. Moreover, the proper choice of $\epsilon$ may vary significantly between two different datasets. 
\item  \underline{Speed}:  The proposed algorithm  is adaptive in the sense that it does not waste computational effort in solving the inner loop to a greater precision than needed. In contrast, the non-adaptive algorithm solves the inner problem to the same degree of precision at every outer step of the algorithm. Overall this results in a significant gain in efficiency.
\end{enumerate}

\section{Notation and Properties of the Algorithm}

In this section we first provide the complete, formalized implementation details for the algorithm described above. We then proceed to develop its mathematical properties.
\subsection{Notation}

First, we recall the definition of the subdifferentials of the TV semi-norm $||f||_{TV}$ and the $\ell_1$ norm $||f||_{1}$ at $f$:
\begin{align}
\partial ||f||_{TV} &:= \left\{ v \in \mathbb{R}^n : ||g||_{TV} - ||f||_{TV} \geq \langle v,g-f \rangle \; \forall g \in \R^n \right\}, \\
\partial ||f||_{1} &:= \left\{v \in \mathbb{R}^n : ||g||_{1} - ||f||_{1} \geq \langle v,g-f \rangle \; \forall g \in \R^n \right\}.
\end{align}
We denote by $\partial_{0} ||f||_{1}$ those elements of the subdifferential $\partial ||f||_{1}$ that have zero mean. As the successive iterates $f^k$ have zero median, $\partial_0 ||f^k||_{1}$ is never empty. For example, we can take $v^k \in \R^n$ so that $v^k(x_i) =1$ if $f(x_i) > 0$,  $v^k(x_i) = -1$ if  $f(x_i) < 0$ and $v^k(x_i) = (n^- -n^+ )/(n_0)$ if $f(x_i) = 0$ where $n^+$, $n^-$ and $n^0$ denote the number of vertices in the sets $\{x_i: f(x_i)>0\}$, $\{x_i: f(x_i)>0\}$ and $\{x_i: f(x_i)=0\}$, respectively.

We next precisely define the approximate total variation step
\begin{equation*}
h^k := \underset{u\in\R^n}{\mathrm{approx} \arg\min}  \; \left\{ ||u||_{TV} +\frac{E(f^k)}{2}||u-g^{k}||_2^2 \right\}
\end{equation*}
that we previously described. From our previous work \cite{BLUV12}, we know that if $H^k$ denotes the (unique) exact solution to the total variation minimization problem,
\begin{equation}
H^k := \underset{u\in\R^n}{\arg\min}  \left\{ ||u||_{TV}+\frac{E(f^k)}{2}||u-g^{k}||_2^2  \right\},
\label{eq:exact}
\end{equation}
then $H^k$ satisfies the energy inequality \eqref{eq:inequality}
\begin{equation}
E(f^k) \geq E(H^k) + \frac{E(f^k) ||H^k - f^k||^2_2}{||H^k - \m(H^k)\ones ||_{1}}.
\end{equation}
In particular, we have that $E(f^k) > E(H^k)$ unless $H^k = f^k$, i.e. $f^k$ itself is the solution to the total variation minimization. In the latter case, it follows from the definition of $H^k$ that there exists $w^k \in \partial ||f^k||_{TV}$ so that
\begin{equation*}
0 = w^k + E(f^k)(f^k - g^k) =  w^k + E(f^k)( f^k - f^k - v^k) = w^k - E(f^k)v^k,
\end{equation*}
which implies the current iterate $f^k$ is a critical point of the energy. 

Turning now to the approximate case, let
\begin{equation*}
\left\{ \Phi^m( E(f^k); g^k ) \right\}^{\infty}_{m=1} \qquad g^k = f^k + v^k \qquad \Phi^1( E(f^k); g^k )  = f^k
\end{equation*}
denote a sequence of iterates that converge to the exact solution starting from the initial point $f^k$, i.e. 
$$
\Phi^m( E(f^k); g^k ) \to H^k \qquad \text{as} \qquad m \to \infty.
$$
In what follows, we use the shorthand $\Phi^m_k$ to denote $\Phi^m( E(f^k); g^k )$. If $H^k \neq f^k,$ the continuity of the energy $E$, the median, the $\ell_1$ norm and the $\ell_2$ norm combine to show that for any $\theta \in (0,1)$ there exists a finite $M_k$ with the following property:
\begin{align*}
E(f^k) &\leq E(\Phi^m_k) +  \frac{\theta \; E(f^k) ||\Phi^m_k - f^k||^2_2}{||\Phi^m_k - \m(\Phi^m_k) \ones ||_1 }   \quad \text{if} \quad 2 \leq m \leq M_k-1 \\
E(f^k) &> E(\Phi^{M_k}_k) + \frac{\theta \; E(f^k) ||\Phi^{M_k}_k - f^k||^2_2}{||\Phi^{M_k}_k - \m(\Phi^{M_k}_k) \ones ||_1 }.
\end{align*}
We can only guarantee such an $M_k$ exists provided $\theta < 1$, so in practice we take a value of $\theta$ close to one, e.g. $\theta = .99$, as we have found this works best in practice. We then define
\begin{align*}
\underset{u\in\R^n}{\mathrm{approx} \arg\min}  \; \left\{ ||u||_{TV} +\frac{E(f^k)}{2}||u-g^{k}||_2^2 \right\} &:= \Phi^{M_k}_{k} \qquad \text{if} \qquad H^k \neq f^k \\
\underset{u\in\R^n}{\mathrm{approx} \arg\min}  \; \left\{ ||u||_{TV} +\frac{E(f^k)}{2}||u-g^{k}||_2^2 \right\} &:= f^{k} \qquad \; \; \text{ if} \qquad H^k = f^k.
\end{align*}
In the second case, i.e. when $f^{k+1} = H^k = f^k$, we terminate the outer loop as well since the algorithm has reached a critical point of the energy. In practice, we set a maximum number of iterations $m \leq M_{\max}$ that, if reached, signifies the ``exact'' solution of \eqref{eq:exact} has been found.

\subsection{Properties of the Approximate Algorithm}
We now proceed to demonstrate that, due to the control afforded us by the energy inequality, the approximate total variation algorithm still enjoys many of the same mathematical properties of the our previous idealized algorithm. We first demonstrate that the intermediate steps $(h^k,h^k_0)$ in the iteration remain in a compact set. If $h^k = f^k$ then obviously $||h^k||_{2} = ||f^k||_{2} = 1$ by definition of the iterates. Otherwise $h^k$ satisfies the energy inequality
\begin{equation*}
E(f^k) > E(h^k) + \frac{\theta \; E(f^k) ||h^k - f^k||^2_2  }{||h^{k} - \m(h^k) \ones ||_1 }.
\end{equation*}
Note that each of the iterates $f^k$ belong to the closed subset
\begin{equation}
\Sn_0 := \{ f \in \R^n : ||f||_2 = 1 \quad \text{and} \quad \m(f) = 0 \}.
\end{equation}
of the $\ell_2$ sphere. As $\Sn_0$ does not contain any constant functions and we assume a connected graph, $E(f)>0$ for all $f \in \Sn_0$. Moreover, since $\Sn_0$ is a closed set on which $E$ is continuous,  $E$ attains a strictly positive minimum $E(f) \geq E(f^*) = \alpha$ on $\Sn_0$, so that $E(f^0) \geq E(f^k) \geq \alpha$ uniformly for all iterates. A combination of this fact with the triangle inequality and the facts that $||x||_{1} \leq \sqrt{n} ||x||_{2}$ and $|\m(x)| \leq ||x||_{2}$ for all $x \in \mathbb{R}^n$ then demonstrates 
\begin{equation*}
||h^k - f^k ||^2_2 \leq \frac{E(f^0) }{\theta \alpha} ||h^{k} - \m(h^k) \ones ||_1 \leq \frac{ E(f^0) }{\theta \alpha} \sqrt{n} (1 + \sqrt{n})||h^k||_{2}.
\end{equation*}
By expanding the inner-product on the left hand side this reveals
\begin{equation*}
||h^k||^2 - 2\langle h^k,f^k \rangle + 1 \leq \frac{E(f^0) }{\theta \alpha} ||h^{k} - \m(h^k) \ones ||_1 \leq \frac{ E(f^0) }{\theta \alpha} \sqrt{n} (1 + \sqrt{n})||h^k||_{2},
\end{equation*}
which by Cauchy-Schwarz implies
\begin{equation*}
||h^k||^{2}_{2} < ||h^k||^2_2 + 1 \leq \left( 2 + \frac{E(f^0) }{\theta \alpha} \sqrt{n} (1 + \sqrt{n}) \right) ||h^k||_{2}.
\end{equation*} 
Dividing by $||h^k||_{2}$ then yields the desired estimate that holds for all $k \geq 0$:
\begin{equation*}
||h^k|| < \left( 2 + \frac{E(f^0) }{\theta \alpha} \sqrt{n} (1 + \sqrt{n}) \right).
\end{equation*} 
In other words, the iterates $h^k$  lies in a fixed, compact set. Arguing as in \cite{BLUV12}, this allows us to obtain

\begin{lemma}[Compactness of $\A_{\mathrm{SD}}$]\label{lem:Comp_SD}
Let $f^0 \in \Sn_0$ and define a sequence of iterates $(g^k,h^k,h^k_0,f^{k+1})$ according to the approximate algorithm. Then there exists an $R>0$ independent of $k$ so that
\begin{equation}
||h^k||_{2} \leq R \quad \text{and} \quad 0 < ||h^k_0||_2 \leq (1+\sqrt{n})||h^k||_2.
\label{upperbound}
\end{equation}
Moreover, we have
\begin{equation} \label{attractor}
||h^k - f^k||_2 \to 0, \qquad \m(h^k) \to 0, \qquad   \|f^k-f^{k+1}\|_2 \to 0.
 \end{equation}
\end{lemma}

\begin{proof}
The first statement follows from the preceeding uniform compactness argument. That $0 < ||h^k_0||_2$ follows since $h^k$ is not constant. Indeed, if $h^k = f^k$ then $h^k \in \Sn_0$ and is therefore not constant. Otherwise, that $h^k$ satisfies the energy inequality implies $||h^k - \m(h^k)\ones||_1 > 0$ and again $h^k$ is not constant. The upper bound $||h^{k}_{0}||_2 \leq (1+\sqrt{n})||h^k||_2$ follows from the triangle inequality. For the second statement, as $f^k \in \Sn_0$ it follows that $E(f^k) \geq \alpha > 0$. From the energy inequality,
\begin{equation} \label{descent2}
 {\| h^k-f^k\|_2^2}  \le\frac{c}{\alpha \theta}  || h^k - \m(h^k)\ones||_{1} (E(f^k)-E(f^{k+1})) \leq C(E(f^k)-E(f^{k+1})) \to 0,
\end{equation}
for some universal constant $C$, due to uniform compactness of the iterates. Convergence to zero follows as $E(f^k)$ is decreasing and bounded from below, and therefore converges.  By continuity of the median and the fact that $\m(f^k) = 0,$ any limit point of the $\{ f^k \}$ must have median zero. As $\| h^k-f^k\|_2^2 \to 0,$ any limit point of the $\{ h^k \}$ must also have median zero, which implies that $\m(h^k) \to 0$ as well. The triangle inequality then implies $||h^k_0 - f^k||_{2} \to 0,$ so that $||h^k_0||_{2} \to 1$ and $||f^{k+1}-f^k||_2 \to 0$ as desired.
\end{proof}

As a consequence of this lemma, we obtain the following corollary that shows the approximate algorithm and the idealized algorithm from \cite{BLUV12} share the same global convergence properties:
\begin{corollary}\label{conv-thm}
Take $f^0 \in \Sn_0$ and let $\{f^k\}$ denote any sequence defined through the approximate total variation algorithm. Then either the sequence $\{f^k \}$ converges or the set of accumulation points form a continuum in $\Sn_0$.
\end{corollary}

\subsection*{The Critical Point Property}
Next, we turn our attention to characterizing the limit points of the sequence $\{f^k\}$. We wish to establish the critical point property, i.e. that any limit point of $\{f^k\}$ is a critical point of the energy. Specifically, if $f^\infty$ denotes a limit point of $\{ f^k \}$ then there exist $v^\infty \in \partial_{0} ||f^\infty||_1$ and $w^\infty \in \partial ||f^\infty||_{TV}$ so that
$$
0 = w^\infty - E(f^\infty)v^\infty.
$$
To this end, let us suppose that we have a subsequence satisfying
\begin{equation*}
f^{k_j} \to f^\infty \qquad f^{k_j +1} \to f^\infty,
\end{equation*}
where the second statement follows from the statement $||f^{k} - f^{k+1}||_{2} \to 0$ in the previous lemma. Note that the previous lemma implies $h^{k_j},h^{k_j}_{0} \to f^\infty$ as well. As $\{v^{k_j} \}$ lie in a uniform compact set, as each entry of $v^{k_j}$ lies in $[-1,1]$, we can (by passing to a further subsequence if necessary) assume that $v^{k_j} \to v^\infty$ for some $v^\infty \in \mathbb{R}^n$. By definition, for all $g \in \R^n$ we have that
\begin{equation*}
||g||_{1} - ||f^{k_j}||_{1} \geq \langle v^{k_j}, g - f^{k_j} \rangle
\end{equation*}
and $\langle v^{k_j},\ones \rangle = 0$, which by passing to the limit $k_j \to \infty$ in both statements reveals that $v^\infty \in \partial_{0} ||f^\infty||_1$ as well.

Before we can establish the critical point property, we clearly must place at least some assumptions on the total variation solver $\Phi^{m}(E(f),g)$. Specifically, we make three assumptions

\begin{assumption}{\rm (Convergence) }
For every $(E(f),g)$ the solver $\Phi^m(E(f),g)$ is convergent, i.e.
$$
\Phi^m(E(f),g) \to \underset{u \in \mathbb{R}^n }{\arg \min} \left\{ ||u||_{TV} + \frac{E(f)}{2} || u - g||^{2}_{2} \right\} \qquad \text{as} \qquad m \to \infty.
$$
\end{assumption}
\begin{assumption}{ \rm (Continuity of the Iterates) } 
For every $m \ge 1$, the function $(E(f),g) \mapsto \Phi^m(E(f),g)$ is continuous.
\end{assumption}
\begin{assumption}{ \rm (The Semigroup Property) } 
For any $m,n \ge 1$, if $\Phi^n(E(f),g)=\Phi^m(E(f),g)$ then $\Phi^{n+1}(E(f),g)=\Phi^{m+1}(E(f),g)$ as well.
\end{assumption}
\noindent We obviously require the first assumption, while the second assumption is reasonable and does in fact hold for the popular total-variation solvers. The third assumption essentially states that during iterative scheme, the next $\Phi^{m+1}$ is determined entirely by the current iterate $\Phi^{m},$ but not by multiple previous iterates or other auxiliary variables. This assumption fails for many of the popular total variation solvers such as the alternating direction method of multipliers or primal-dual algorithms. It does hold for so-called ``first-order'' solvers, however, such as straightforward gradient-descent, forward-backward splitting schemes or Uzawa  iteration applied to the dual problem. We include it for simplicity in illustrating that, as a proof-of-concept, the control afforded us by the energy inequality allows us to retain in the approximate algorithm \emph{all} convergence properties of the idealized algorithm. We leave the proof in the more general case to future work.

Returning now to establishing the critical point property, assume that $f^\infty$ is not a critical point of the energy. By definition, then,
\begin{equation*}
0 \notin \partial ||f^\infty||_{TV} - E(f^\infty)v^\infty \qquad \Leftrightarrow \qquad 0 \notin \partial ||f^\infty||_{TV} + E(f^\infty)( f^\infty - g^\infty), \qquad g^\infty = f^\infty + v^\infty.
\end{equation*}
In particular,
$$
f^\infty \neq \underset{u\in\R^n}{\arg\min}  \; \left\{ ||u||_{TV}+\frac{E(f^\infty)}{2}||u-g^{\infty}||_2^2 \right\}.
$$
As before define
\begin{equation*}
H^\infty := \underset{u\in\R^n}{\arg\min}  \; \left\{ ||u||_{TV} +\frac{E(f^\infty)}{2}||u-g^{\infty}||_2^2 \right\}
\end{equation*}
along with the corresponding sequence of iterates
\begin{equation*}
\Phi^m(E(f^\infty),g^\infty) \to H^\infty \qquad \text{as} \qquad m \to \infty, \qquad \Phi^1(E(f^\infty),g^\infty) = f^\infty.
\end{equation*}
As $f^\infty \neq H^\infty$ there exists a finite $M$ with the property that (where $\Phi^m_{\infty}$ is shorthand for $\Phi^m(E(f^\infty),g^\infty)$
\begin{align*}
E(f^\infty) &\leq E(\Phi^m_{\infty}) + \frac{\theta \; E(f^\infty) || \Phi^m_{\infty} - f^\infty||^{2}_{2} }{ || \Phi^m_{\infty} - \m(\Phi^m_{\infty})\ones ||_{1} } \quad \text{if} \quad m \leq M-1 \\
E(f^\infty) &> E(\Phi^M_{\infty}) + \frac{\theta \; E(f^\infty) || \Phi^M_{\infty} - f^\infty||^{2}_{2} }{ || \Phi^M_{\infty} - \m(\Phi^M_{\infty})\ones ||_{1} } 
\end{align*}
We may suppose that each of the iterates $h^{k_j}$ came from an approximate total variation solve, i.e. $h^{k_j} = \Phi^{M_{k_j}}(E(f^{k_j}),g^{k_j})$ for some finite iteration number $M_{k_j}$, since if this is not the case then the sequence $\{f^k\}$ reaches a critical point of the energy in a finite number of iterations. 

As $f^{k_j} \to f^\infty$, $E(f^{k_j}) \to E(f^\infty)$ and $g^{k_j} \to g^{\infty}$ and the approximate total variation procedure performed at $(E(f^\infty)^\infty, g^\infty)$ terminates in $M$ iterations, we would expect that for $j$ large enough the approximate  total variation procedure at $(E(f^{k_j}),g^{k_j})$ would also terminate in $M$ iterations but here we must be a bit more careful. By the continuity of the iterates $\Phi^m$ we do have that the energy inequality
\begin{align*}
E(f^{k_j}) &> E(\Phi^M_{k_j} ) + \frac{\theta \; E(f^{k_j}) ||\Phi^M_{k_j} - f^{k_j}||^2_2}{|| \Phi^M_{k_j} - \m( \Phi^M_{k_j}) \ones ||_{1}  }
\end{align*}
holds for all $k_j$ sufficiently large. In other words, there exists $J$ so that if $j\geq J$ then $M_{j} \leq M$. As $M_{j} \leq M$ for all $k_j$ sufficiently large, this means that the entire sequence $\{ M_{j} \}^{\infty}_{j=1}$ is, in fact, bounded. We may therefore extract yet another subsequence $k_{j_l}$ so that $M_{j_l} \to M^*$ for some $2\leq M^* \in \mathbb{R}$. However, as the $M_{j_l}$ form a Cauchy sequence and are also integers (so, $|M_{j_l} - M_{j_{l^\prime}}| \geq 1$ unless they are equal) this implies that in fact $M_{j_l} \equiv M^* \in \mathbb{N}$ for all $l$ sufficiently large. So along this subsequence we also have
$$
h^{k_{j_l}} = \Phi^{M^*}_{{k_{j_l}}} , \qquad h^{k_{j_l}},h^{k_{j_l}}_{0},f^{k_{j_l}+1}  \to f^\infty, \qquad v^{k_{j_l}} \to v^\infty.
$$
That is, for $l$ large enough the terminating index does not change. As $h^{k_{j_l}} \to f^\infty$, it follows from continuity of the iterates that
$$
\Phi^{M^*}(E(f^\infty),g^\infty) = f^\infty = \Phi^{1}(E(f^\infty),g^\infty).
$$
By the semigroup property, for any $n \in \mathbb{N}$ it follows that $\Phi^{1 + n(M^*-1)}(E(f^\infty),g^\infty) = f^\infty$ as well. In particular, $f^\infty$ appears infinitely often. As $\Phi^m$ converges as $m \to \infty$, we then necessarily have  $\Phi^m(E(f^\infty),g^\infty) =f^\infty$ for all $m$ and $\lim_{m\to \infty}\Phi^m(E(f^\infty),g^\infty)=f^{\infty}$, that is:
\begin{equation}
f^\infty = \underset{u\in\R^n}{\arg\min}  \; \left\{ ||u||_{TV}+\frac{E(f^\infty)}{2}||u-g^{\infty}||_2^2 \right\}.
\label{eq:not_cv}
\end{equation}
This contradicts the assumption that $f^\infty$ is not a critical point of the energy, which completes the proof.

\begin{remark}
While the semigroup assumption suffices to establish the critical point property, it often proves too restrictive. If instead we establish the existence of a strictly monotone quantity $F$, i.e. $F(\Phi^{m}) > F(\Phi^{m+1})$ unless $\Phi^{m} = \Phi^{m+1}$, such as the total variation energy or the residual then the same proof works even in the absence of the semigroup property.
\end{remark}

\section{A stopping condition for the inner TV problem which does not involve computing the median}
In this section we present an alternative approximate total variation algorithm that avoids having to compute the energy $E(\Phi^{m})$ at each iteration of the total variation solver. The motivation for this lies in the fact that other total variation clustering problems, such as TV-Normalized Cut, rely on energies with weighted medians that are expensive to compute. An algorithm that avoids this extra computation, yet still satisfies the energy inequality, would therefore produce an additional gain in efficiency. We develop this idea for the TV-Balanced Cut problem; the idea extends in a straightforward fashion to other total variation clustering problems.

If we solve the inner total variation problem exactly, i.e. we compute
$$
H^k := \underset{ u \in \mathbb{R}^n }{\arg \min} \left\{ ||u||_{TV} + \frac{E(f^k)}{2}||u - g^k||^2_2 \right\},
$$
then we have that
$$
E(f^k)\left(g^k - H^k\right) \in \partial ||H^k||_{TV}.
$$
In particular, this implies that
$$
||f^k||_{TV} \geq ||H^k||_{TV} + E(f^k)\langle g^k - H^k,f^k - H^k \rangle = ||H^k||_{TV} + E(f^k)|| H^k - f^k||^2_2  - E(f^k) \langle v^k,H^k - f^k\rangle.
$$
Now let $\theta = .99$ and $\Phi^{m}_{k} := \Phi^{m}(E(f^k),g^k) \to H^k$ when $m \to \infty$ as before. If $H^k = f^k$ then $f^k$ is a critical point of the energy and we terminate the algorithm. Otherwise $H^k \neq f^k$ so there exists a finite $M_k$ with the property that
\begin{align}
||f^k||_{TV} &\leq ||\Phi^{m}_{k} ||_{TV} + \theta \; E(f^k)|| \Phi^{m}_{k} - f^k||^2_2  - E(f^k) \langle v^k,\Phi^{m}_{k} - f^k\rangle \qquad 2 \leq m \leq M_k-1 \\
||f^k||_{TV} &> ||\Phi^{M_k}_{k}||_{TV} + \theta \; E(f^k)|| \Phi^{M_k}_{k} - f^k||^2_2  - E(f^k) \langle v^k,\Phi^{M_k}_{k} - f^k\rangle, 
\label{eq:new_cond}
\end{align}
and we set $h^k  = \Phi^{M_k}_{k}$ just as in the previous algorithm. Note that checking \eqref{eq:new_cond} only requires computing $||\Phi^{m}_{k} ||_{TV}$ and two inner-products at each iteration. It then follows, due to the fact that $v^k \in \partial_0 ||f^k||_{1}$, that
\begin{equation*}
||h^k - \m(h^k) \ones ||_{1} - ||f^k||_{1} \geq \langle v^k, h^k - f^k \rangle.
\end{equation*}
Multiplying this inequality by $E(f^k)$ and adding it to the previous inequality yields
\begin{equation*}
||f^k||_{TV} + E(f^k)||h^k - \m(h^k) \ones ||_{1}  > E(f^k) ||f^k||_{1} + ||h^k||_{TV} + \theta \; E(f^k)|| h^k - f^k||^2_2,
\end{equation*}
or in other words
\begin{equation*}
E(f^k) > E(h^k) + \frac{\theta \; E(f^k)|| h^k - f^k||^2_2 } { ||h^k - \m(h^k) \ones ||_{1}  }.
\end{equation*}
That is, $h^k$ satisfies the desired energy inequality. As a consequence, all of the compactness and convergence results from the previous section hold, with only slight modification, for this algorithm as well.

Using \eqref{eq:new_cond} as a stopping condition in the total variation minimization solver  leads to the following variation of Algorithm \ref{algo2}:

\begin{algorithm}
\caption{Variation of Algorithm \ref{algo2} without median in inner stopping condition}
\vspace{.1cm}
\begin{center}
\begin{algorithmic}
\STATE $f^{0} $ non-constant function with $\m(f) = 0$ and $||f^0||_{2} = 1$, $\theta=.99$. \\
\WHILE{$E(f^{k}) - E(f^{k+1})  \geq \mathrm{TOL}$ }
\STATE $v^k \in \partial_{0} ||f^k||_{1}$ \\
\STATE{{$g^{k}=f^{k}+\, v^k$}}\\
\STATE{ Solve $h^k \approx \underset{u\in\R^n}{\arg\min}  \{ ||u||_{TV}+\frac{E(f^k)}{2}||u-g^{k}||_2^2  \}$ until 
	$$||f^k||_{TV} > ||h^k||_{TV} + \theta \; E(f^k)||h^k - f^k||^2_2  - E(f^k) \langle v^k,h^k - f^k\rangle, $$ }\\
\STATE{$h^{k}_{0} =h^k - \m(h^k)\ones$}
 \STATE{$f^{k+1}= \frac{h^k_0}{\|h^k_0\|_2} $}
\ENDWHILE
\end{algorithmic}
\end{center}
\label{algo3}
\end{algorithm}

\subsection*{Local Convergence Results}
By leveraging the inequality \eqref{eq:new_cond}, we can demonstrate that this approximate algorithm satisfies the same local convergence properties as the idealized algorithm. Recalling the definition from [NIPS], we say that a set-valued algorithm $\mathcal{A}$ is \emph{closed at local minima} (the CLM property) if $f^k \to f^\infty \in \Sn_0$ and $z^k \in \mathcal{A}(f^k)$ then $z^k \to f^\infty$ whenever $f^\infty$ is a local minimum of the energy. Note that the approximate algorithm defined above is, in fact, a set-valued algorithm due to the lack of uniqueness in $v^k$, i.e. the choice of subdifferential.

To demonstrate the CLM property for the approximate algorithm, suppose we have a sequence $f^k \in \Sn_0$ converging to some $f^\infty \in \Sn_0$ and let $h^k$ denote the corresponding sequence of intermediate steps. If $h^k \neq f^k$ only finitely many times then the CLM property is immediate. Indeed, then $h^k = f^k$ for all $k$ sufficiently large, which implies $h^k \to f^\infty$, $h^k_0 \to f^\infty$ and $z^k := h^k_0 / ||h^k_0||_2 \to f^\infty$ as well. Otherwise, $h^k \neq f^k$ infinitely many times. Given any subsequence of $\{z^k\}$ we may restrict attention to a further subsequence for which $h^{k_j} \neq f^{k_j}$ along the entire subsequence. As the $h^{k_j}$ satisfy the energy inequality, they lie in a compact set. By passing to a further subsequence if necessary, we may therefore assume that $h^{k_j} \to h^{\infty}$ and  $v^{k_j} \to v^\infty \in \partial_{0} ||f^\infty||_{1}$ while still retaining $f^{k_j} \to f^\infty$ and the fact that $h^{k_j}$ satisfy \eqref{eq:new_cond}. 

We now suppose that $h^\infty \neq f^\infty$ and shall obtain a contradiction. Indeed, if $h^\infty \neq f^\infty$ then by passing to the limit we find 
\begin{equation}
||f^\infty||_{TV} \geq ||h^\infty||_{TV} + \theta \; E(f^\infty)||h^\infty - f^\infty||^2_2  - E(f^\infty) \langle v^\infty,h^\infty - f^\infty\rangle. 
\label{eq:airhockey}
\end{equation}
For $\eta \in (0,1)$ let $h_{\eta} := \eta h^\infty + (1-\eta) f^\infty$. By convexity of the TV semi-norm, 
$$
||h_{\eta}||_{TV} \leq \eta ||h^\infty||_{TV} + (1-\eta) ||f^\infty||_{TV} \quad \Rightarrow \frac{1}{\eta} ||h_{\eta}||_{TV} - \frac{1-\eta}{\eta}||f^\infty||_{TV} \leq ||h^\infty||_{TV}.
$$
Substituting this estimate into \eqref{eq:airhockey} then shows
$$
\frac{1}{\eta} ||f^\infty||_{TV} \geq \frac{1}{\eta} ||h_\eta||_{TV} + \theta \; E(f^\infty)||h^\infty - f^\infty||^2_2  - \frac{1}{\eta} E(f^\infty) \langle v^\infty,h_\eta - f^\infty\rangle.
$$
Once again, the fact that $v^\infty \in \partial_{0} ||f^\infty||_{1}$ implies
$$
||h_{\eta} - \m(h_\eta)\ones||_{1} \geq ||f^\infty||_{1} + \langle v^\infty,h_\eta - f^\infty\rangle.
$$
Multiplying this inequality by $E(f^\infty)$ and adding it to $\eta$ times the previous inequality then shows
$$
E(f^\infty)||h_{\eta} - \m(h_\eta)\ones||_{1} \geq ||h_\eta||_{TV} + \eta \theta E(f^\infty)||h^\infty - f^\infty||^2_2.
$$
We may assume that $h_{\eta}$ is not constant, since otherwise this would imply $h^\infty = f^\infty$ as desired. We may therefore divide by $||h_{\eta} - \m(h_\eta)\ones||_{1}$ in the previous inequality to obtain
$$
E(f^\infty) \geq E(h_\eta) + \eta \theta E(f^\infty) \frac{ ||h^\infty - f^\infty||^2_2 } {||h_{\eta} - \m(h_\eta)\ones||_{1} }.
$$
If $||h^\infty - f^\infty||^2_2 > 0$ this would imply $E(h_\eta) < E(f^\infty)$ for any $\eta$ that is strictly positive. As $h_\eta \to f^\infty$ as $\eta \to 0$ this would contradict the fact that $f^\infty$ is a local minimum of the energy, whence $h^\infty = f^\infty$ as desired. Thus any subsequence of $h^k$ has a further subsequence that converges to $f^\infty$, meaning the whole sequence converges to this limit. This then implies that $h^k_0 \to f^\infty$ and $z^k \to f^\infty$ as well, and this establishes the CLM property for the approximate algorithm.

To formulate a notion of local convergence, we need an analogue of a ``strict'' local minimum of the TV-Balanced Cut energy. Due to the invariance of this energy under scaling and the addition of constants, we cannot refer to a local minimum as ``strict'' in the usual sense.  We must therefore remove the effects of these invariances when referring to a local minimum as strict. To this end, define the spherical and annular neighborhoods on $\Sn_0$ by
\begin{align*}
	\mathcal{B}_{\epsilon}(f^\infty) := \left\{ ||f-f^\infty||_2 \leq \epsilon \right\} \cap \Sn_0 \quad \mathcal{A}_{\delta,\epsilon}(f^\infty) := \left\{ \delta \leq ||f-f^\infty||_2 \leq \epsilon \right\} \cap \Sn_0.
\end{align*}
With these in hand we introduce the proper definition of a strict local minimum.
\begin{definition}[Strict Local Minima]
Let $f^\infty \in \Sn_0$. We say $f^\infty$ is a \textbf{strict local minimum} of the energy if there exists $\epsilon > 0$ so that $f \in \mathcal{B}_\epsilon(f^\infty)$ and $f \neq f^\infty$ imply $E(f) > E(f^\infty)$.
\end{definition}
\noindent The CLM property now allows us to quote a general result from \cite{BLUV12} that establishes a local stability property for the approximate algorithm:
\begin{lemma}[Lyapunov Stability at Strict Local Minima]\label{lem:stability}
Fix $f^0 \in \Sn_0$ and let $\{ f^k \} $ denote any sequence corresponding to the approximate algorithm. If $f^\infty$ is a strict local minimum of the energy, then for any $\epsilon > 0$ there exists a $\gamma > 0$ so that if $f^0 \in \mathcal{B}_{\gamma}(f^\infty)$ then $\{ f^k \} \subset \mathcal{B}_{\epsilon}(f^\infty)$.
\end{lemma}

Loosely speaking, this means that if we have a good initial guess for the solution of the TV-Balanced Cut problem then the approximate algorithm defined above will remain close to this initial guess while simultaneously lowering the TV-Balanced Cut energy. We emphasize that this property holds regardless of any assumptions made about the total variation solver $\Phi^{m}$ other than convergence, e.g. the semigroup property. If we further assume the continuity and semigroup properties of the solver then this approximate algorithm satisfies the critical point property as well. In this case, the remaining theory of \cite{BLUV12} applies and we do, in fact, recover precisely \emph{all} of the theoretical properties of the idealized algorithm with this approximate total variation algorithm.

\section{Numerical Experiments}
All experiments that follow use a symmetric $k$-nearest neighbor graph combined with the weight similarity function $w_{i,j}= \exp(-r_{i,j}^2/\sigma^2)$. Here, $r_{i,j} = \|x_i -x_j \|_2$ and the scale parameter $\sigma^2=3d_k^2$, where $d_k$ denotes the mean distance of the $k^{th}$ nearest neighbor.

We use the two-moon, MNIST and USPS datasets. The two-moon dataset  \cite{pro:BuhlerHein09pLapla} uses the same setting as in  \cite{pro:SzlamBresson10}. We take $k = 5$ nearest neighbors to construct the graph. We preprocessed the MNIST and USPS data by projecting onto the first $50$ principal components, and take $k = 10$ nearest neighbors for the MNIST and USPS datasets.

We use Algorithm \ref{algo3} and
 the method from \cite{art:ChambollePock11FastPD} to solve  the inner ROF  problem \eqref{TVstep}. We terminate each inner loop when either the condition
$$||f^k||_{TV} > ||h^k||_{TV} + \theta \; E(f^k)||h^k - f^k||^2_2  - E(f^k) \langle v^k,h^k - f^k\rangle, $$
  is satisfied or $1,500$ iterations is reached (meaning that the solution has been found). We take $\theta=0.99$ is all experiments.

The following table summarizes the results of these tests. It shows  the mean error of classification (\% of misclassified data) and the mean computational time for the proposed algorithm and the previous algorithm from \cite{BLUV12} over $10$ experiments.\\

\vspace{-.2cm}
\begin{table}[h!]
\centering
\begin{tabular}{|c|c|c|c|c|c|c|c|c|c|}
\cline{2-5}		
 \multicolumn{1}{}{} &  \multicolumn{2}{|c|}{ Adaptive Algorithm \ref{algo3}} & \multicolumn{2}{|c|}{Non-adaptive Algorithm from   \cite{BLUV12}} \\
\cline{2-5}	
\multicolumn{1}{c|}{}  & Error (\%) & Time  & Error (\%) & Time \\
\hline	
2 moons &  9.06 & 2.03 sec. & 8.69 & 2.06 sec.   \\
\hline	
MNIST  (10 classes) &  11.76 & 21.85 min. & 11.78 & 45.01 min. \\
\hline	
USPS  (10 classes) &  4.11 & 3.08 min. & 4.11 & 5.15 min. \\
\hline		
\end{tabular}
\label{tab:2classes}
\end{table}


{\bf Reproducible research:} The code is available at \href{http://www.cs.cityu.edu.hk/~xbresson/codes.html}{http://www.cs.cityu.edu.hk/$\sim$xbresson/codes.html}\\

{\bf Acknowledgements:} This work supported by AFOSR MURI grant FA9550-10-1-0569, NSF grant  DMS-0902792,  and Hong Kong GRF grant \#110311.

\bibliographystyle{plain}

\bibliography{bib_nips}

\end{document}